\documentclass[11pt, a4paper]{article}
\usepackage{amsmath}
\usepackage{wasysym}
\usepackage{latexsym}
\usepackage{amsfonts}
\usepackage{mathrsfs}
\usepackage{amssymb}
\usepackage{ifsym}
\usepackage{dsfont}
\usepackage[all]{xy}

\newtheorem{theorem}{Theorem}[section]

\newtheorem{definition}[theorem]{Definition}
\newtheorem{lemma}[theorem]{Lemma}
\newtheorem{corollary}[theorem]{Corollary}

\newtheorem{question}[theorem]{Question}
\newtheorem{remark}[theorem]{Remark}
\newtheorem{observation}[theorem]{Observation}

\newenvironment{proof}[1][Proof]{\begin{trivlist}
\item[\hskip \labelsep {\bfseries #1}]}{\end{trivlist}}

\newcommand{\qed}{\nobreak \ifvmode \relax \else
      \ifdim\lastskip<1.5em \hskip-\lastskip
      \hskip1.5em plus0em minus0.5em \fi \nobreak
      \vrule height0.75em width0.5em depth0.25em\fi}

\begin{document}
\title{The set-theoretic Kaufmann-Clote question} 
\author{Zachiri McKenzie\\
University of Chester\\
{\tt z.mckenzie@chester.ac.uk}}
\maketitle

\begin{abstract}
Let $\mathsf{M}$ be the set theory obtained from $\mathsf{ZF}$ by removing the collection scheme, restricting separation to $\Delta_0$-formulae and adding an axiom asserting that every set is contained in a transitive set. Let $\Pi_n\textsf{-Collection}$ denote the restriction of the collection scheme to $\Pi_n$-formulae. In this paper we prove that for $n \geq 1$, if $\mathcal{M}$ is a model of $\mathsf{M}+\Pi_n\textsf{-Collection}+\mathsf{V=L}$ and $\mathcal{N}$ is a $\Sigma_{n+1}$-elementary end extension of $\mathcal{M}$ that satisfies $\Pi_{n-1}\textsf{-Colelction}$ and that contains a new ordinal but no least new ordinal, then $\Pi_{n+1}\textsf{-Collection}$ holds in $\mathcal{M}$. This result is used to show that for $n \geq 1$, the minimum model of $\mathsf{M}+\Pi_n\textsf{-Collection}$ has no $\Sigma_{n+1}$-elementary end extension that satisfies $\Pi_{n-1}\textsf{-Collection}$, providing a negative answer to the generalisation of a question posed by Kaufmann.  
\end{abstract}

\section[Introduction]{Introduction}

Paris and Kirby \cite{pk78} reveal the connection between fragments of the arithmetic collection scheme and the existence of proper $\Sigma_n$-elementary end extensions of models of a weak subsystem of arithmetic. They show that if $\mathcal{M}$ is a countable model of $I\Delta_0$, then for all $n \in \omega$ with $n \geq 1$, $\mathcal{M}$ satisfies $B\Sigma_{n+1}$ if and only if $\mathcal{M}$ has a proper $\Sigma_{n+1}$ elementary end extension. Here $I\Delta_0$ is the subsystem of Peano Arithmetic obtained by restricting induction to bounded formulae, and $B\Sigma_n$ is obtained from $I\Delta_0$ by adding the arithmetic collection scheme for $\Sigma_n$-formulae. Let $\mathsf{M}^-$ be set theory obtain from {\it Kripke-Platek Set Theory} by weakening the foundation available to only {\it set foundation}, removing $\Delta_0\textsf{-Collection}$ and adding an axiom asserting that every set is contained in a transitive set (see section \ref{Sec:Background} below for details). In \cite[Theorem 1]{kau81}, Kaufmann proves the following set-theoretic analogue of Paris and Kirby's result:

\begin{theorem} \label{Th:KaufmannTheorem}
(Kaufmann) Let $n \in \omega$ with $n \geq 1$. Let $\mathcal{M}= \langle M, \in \rangle$ be a model of $\mathsf{M}^-$. Consider 
\begin{itemize}
\item[(I)] there exists $\mathcal{N}= \langle N, \in^{\mathcal{N}} \rangle$ such that $\mathcal{M} \prec_{e, n+1} \mathcal{N}$ and $M\neq N$;
\item[(II)] $\mathcal{M} \models \Pi_n\textsf{-Collection}$. 
\end{itemize}
If $\mathcal{M} \models \mathsf{V=L}$, then $(I) \Rightarrow (II)$. And, if $M$ is countable, then $(II) \Rightarrow (I)$. \Square 
\end{theorem}

In \cite{kau81}, it is not assumed that $\mathcal{M}$ satisfies set foundation, and $(I) \Rightarrow (II)$ is proved from the weaker assumption that $\mathcal{M}$ is {\it resolvable}. Theorem \ref{Th:KaufmannTheorem} above also reformulates \cite[Theorem 1]{kau81} by making use of the equivalence of collection for $\Pi_n$-formulae with collection for $\Sigma_{n+1}$-formulae.

A result due to Simpson (see \cite[Remark 2]{kau81}) shows that for countable limit ordinals $\alpha$, $\langle L_\alpha, \in \rangle$ has a $\Sigma_1$-elementary end extension $\mathcal{N}$ that satisfies Kripke-Platek Set Theory plus $\mathsf{V=L}$ with a new ordinal but no least new ordinal if and only if $\langle L_\alpha, \in \rangle$ satisfies $\Pi_1\textsf{-Collection}$. A consequence is that there is a transitive model of Kripke-Platek Set Theory with no proper $\Sigma_1$-elementary end extension that satisfies Kripke-Platek Set Theory. Using a generalisation of Simpson's result, \cite[\S 3]{mck25} shows that there are transitive models of $\Pi_n\textsf{-Collection}$ that have no proper $\Sigma_{n+1}$-elementary end extensions that satisfy $\Pi_n\textsf{-Collection}$ (indeed, these transitive models can even satisfy powerset and full separation). Kaufmann \cite[p. 102]{kau81} asks:

\begin{question} \label{Q:KaufmannQuestion}
If $\langle L_\alpha, \in \rangle$ has a proper $\Sigma_2$-elementary extension, does it necessarily have one satisfying $\Delta_0\textsf{-Collection}$?  
\end{question}

Using Theorem \ref{Th:KaufmannTheorem}, the restriction of Question \ref{Q:KaufmannQuestion} to countable ordinals, $\alpha$, asks: Does every $\langle L_\alpha, \in \rangle$ that satisfies $\Pi_1\textsf{-Collection}$ have a proper $\Sigma_2$-elementary elementary end extension that satisfies $\Delta_0\textsf{-Collection}$? Clote \cite[p. 39]{clo85} (see also \cite[Problem 33]{ck93}) asks a generalised version of Kaufmann's Question in the context of arithmetic: For $n \in \omega$ with $n \geq 1$, does every countable model of $B\Sigma_{n+1}$ have proper $\Sigma_{n+1}$-elementary end extension that satisfies $B\Sigma_n$? This suggests the following set-theoretic generalisation of Question \ref{Q:KaufmannQuestion}:

\begin{question} \label{Q:GeneralisedKaufmannCloteQuestion}
(Set-theoretic Kaufmann-Clote Question) Let $n \in \omega$ with $n \geq 1$. Does every $\langle L_\alpha, \in \rangle$ such that $\alpha$ is countable and $\langle L_\alpha, \in \rangle$ satisfies $\Pi_n\textsf{-Collection}$ have a proper $\Sigma_{n+1}$-elementary end extension that satisfies $\Pi_{n-1}\textsf{-Collection}$? 
\end{question}

Let $\mathsf{M}$ be the set theory obtain by adding the powerset axiom to $\mathsf{M}^-$. This article provides a negative answer to Question \ref{Q:GeneralisedKaufmannCloteQuestion} (and, in particular, a negative answer to Question \ref{Q:KaufmannQuestion}) by showing that if $n \in \omega$ with $n \geq 1$ and $\alpha$ is the least ordinal such that $\langle L_\alpha, \in \rangle$ satisfies $\mathsf{M}+\Pi_n\textsf{-Collection}$, then $\langle L_\alpha, \in \rangle$ has no proper $\Sigma_{n+1}$-elementary end extension that satisfies $\Pi_{n-1}\textsf{-Collection}$. This result is obtain from the following strong variant of Simpson's result \cite[Remark 2]{kau81}:
   
\begin{theorem}[Main Theorem] \label{Th:MainTheorem}
Let $n \in \omega$ with $n \geq 1$. Let $\mathcal{M}$ be a model of $\mathsf{M}+\Pi_n\textsf{-Collection}+\mathsf{V=L}$. If $\mathcal{N}$ is such that $\mathcal{N}$ is a $\Sigma_{n+1}$-elementary end extension of $\mathcal{M}$ that satisfies $\Pi_{n-1}\textsf{-Collection}$ and $\mathcal{N}$ contains a new ordinal but no least new ordinal, then $\mathcal{M}$ satisfies $\Pi_{n+1}\textsf{-Collection}$. \Square
\end{theorem}

In contrast, recent work of Sun \cite{sunXX} shows that if $n \in \omega$ with $n \geq 1$, then every countable model of $B\Sigma_{n+1}$ has a proper $\Sigma_{n+1}$-elementary end extension that satisfies $B\Sigma_n$, providing a positive answer to Clote's question \cite[p. 39]{clo85}.

Section \ref{Sec:ComplexityCalculations} analyses the theory that $\Sigma_{n+1}$-elementarity transfers to a $\Sigma_{n+1}$-elementary end extension of a structure that satisfies $\mathsf{M}+\Pi_n\textsf{-Collection}$, showing that additional collection holding in the end extension allows us to transfer more foundation from the structure being extended. Section \ref{Sec:BaseCase} proves the base case ($n=1$) of the Main Theorem, and section \ref{Sec:GeneralCase} proves the Main Theorem for $n \geq 2$.

\section[Background]{Background} \label{Sec:Background}

Throughout this paper $\mathcal{L}$ will denote the language of set theory- first-order logic with equality ($=$) and a binary relation symbol $\in$. As usual, $\Delta_0(=\Sigma_0=\Pi_0)$, $\Sigma_1$, $\Pi_1$, \ldots denote the L\'{e}vy classes $\mathcal{L}$-formulae. Takahashi \cite{tak72} introduces the class $\Delta_0^{\mathcal{P}}$ extending $\Delta_0$ that consists of all $\mathcal{L}$-formulae whose quantifiers are bounded either by the membership relation ($\in$) or the subset relation ($\subseteq$). Let $T$ be a theory in a language that contains $\mathcal{L}$ and let $\Gamma$ be a collection of $\mathcal{L}$-formulae. We use $\Gamma^T$ to denote the collection of $\mathcal{L}$-formulae that $T$ proves are equivalent to a formula in $\Gamma$. We use $\Delta_n^T$ to denote the collection of formulae that $T$ proves are equivalent to both a $\Sigma_n$-formula and a $\Pi_n$-formula.  

We use \textsf{Separation} and \textsf{Collection} to denote the usual set-theoretic separation and collection schemes, respectively, that are used to axiomatise $\mathsf{ZF}$. \textsf{Foundation} denotes the theorem scheme of $\mathsf{ZF}$ asserting that every non-empty definable (with parameters) class has an $\in$-minimal element. If $\Gamma$ is a collection of formulae, the we use $\Gamma\textsf{-Separation}$, $\Gamma\textsf{-Collection}$ and $\Gamma\textsf{-Foundation}$ to denote the restrictions of the schemes \textsf{Separation}, \textsf{Collection} and \textsf{Foundation}, respectively, to formulae in $\Gamma$. We use  \textsf{Set-Foundation} to denote the single axiom asserting that every set has an $\in$-minimal element. 

Let $\mathsf{M}^-$ be the $\mathcal{L}$-theory with axioms:
\begin{center}
\begin{tabular}{ll}
(\textsf{Extensionality}) & $\forall x \forall y (x=y \iff \forall z(z \in x \iff z \in y)$;\\
(\textsf{Pair}) & $\forall x \forall y \exists z(x \in z \land y \in z \land (\forall w \in z)(w=x \lor w=y))$;\\
(\textsf{Union}) & $\forall x \exists y((\forall z \in x)(\forall w \in z)(w \in y) \land (\forall w \in y)(\exists z \in x)(w \in z))$;\\
(\textsf{Emptyset}) & $\exists x (\forall y \in x)(y \neq y)$;\\
(\textsf{Infinity}) & $\exists x \left( \emptyset \in x \land (\forall y \in x)(\exists z \in x) \left(\begin{array}{c}
(\forall w \in z)(w = y \lor w \in y) \land\\
(\forall w \in y)(w \in z) \land y \in z
\end{array} \right)\right)$;\\
(\textsf{TCo}) & $\forall x \exists y ((\forall z \in x)(z \in y) \land (\forall z \in y)(\forall w \in z)(w \in y))$;\\
(\textsf{Set-Foundation}) & $\forall x((\exists w \in x)(w=w) \Rightarrow (\exists y \in x)(\forall z \in y)(z \notin x))$;\\
($\Delta_0\textsf{-Separation}$) & for all $\Delta_0$-formulae, $\phi(x, \vec{z})$,\\ 
& $\forall \vec{z} \forall w \exists y ((\forall x \in y)(x \in w \land \phi(x, \vec{z})) \land (\forall x \in w)(\phi(x, \vec{z}) \Rightarrow x \in y))$.
\end{tabular}
\end{center}  
This makes it clear that $\mathsf{M}^-$ is axiomatised by a collection of $\Pi_2$-sentences. The theory $\mathsf{M}$ is obtained from $\mathsf{M}^-$ by adding the $\Pi_3$-sentence
\begin{itemize}
\item[](\textsf{Powerset})\ \ \ $\forall x \exists y \forall z(z \in y \iff z \subseteq x)$.
\end{itemize}  
{\it Kripke-Platek Set Theory ($\mathsf{KP}$)} is obtained from $\mathsf{M}^-$ by removing \textsf{Infinity} and adding $\Delta_0\textsf{-Collection}$ and $\Pi_1\textsf{-Foundation}$. The theory $\mathsf{KPI}$ is obtained from $\mathsf{KP}$ by adding \textsf{Infinity}. This  differs from \cite{bar75} and \cite{fri73} where Kripke-Platek Set Theory is defined to include \textsf{Foundation} instead of the more restrictive $\Pi_1\textsf{-Foundation}$. It should also be noted that $\mathsf{TCo}$ is not usually included as an axiom of either $\mathsf{KP}$ or $\mathsf{KPI}$. However, as can be seen from verifying that the proof of \cite[Theorem I.6.1]{bar75} uses no more than $\Pi_1\textsf{-Foundation}$, $\mathsf{TCo}$ is a consequence of the other axioms of $\mathsf{KP}$ making its addition here superfluous. 

The presence of $\mathsf{TCo}$ and \textsf{Set-Foundation} in $\mathsf{M}^-$ means that for any collection of formulae $\Gamma$, $\Gamma\textsf{-Separation}$ implies $\Gamma\textsf{-Foundation}$. The axiom $\mathsf{TCo}$ also facilitates the equivalence between \textsf{Foundation} and set induction. In the theory $\mathsf{M}^-$, $\Gamma\textsf{-Foundation}$ is equivalent to set induction for all formulae in $\neg \Gamma= \{ \neg \phi \mid \phi \in \Gamma \}$. It is well known (by generalising the proof \cite[Theorem I.4.4]{bar75}, for example) that, over the theory $\mathsf{M}^-$, $\Pi_n\textsf{-Collection}$ is equivalent to $\Sigma_{n+1}\textsf{-Collection}$. \cite[Theorem 4.13]{flw16} shows that the theory $\mathsf{M}^-+\Pi_n\textsf{-Collection}$ proves the following separation scheme:
\begin{itemize}
\item[]($\Delta_{n+1}\textsf{-Separation}$) for all $\Sigma_{n+1}$-formuale, $\phi(x, \vec{z})$, and for all $\Pi_{n+1}$-formulae, $\psi(x, \vec{z})$,
\[
\forall \vec{z}(\forall v(\phi(v, \vec{z}) \iff \psi(v, \vec{z})) \Rightarrow \forall w \exists y \forall x(x \in y \iff x \in w \land \phi(x, \vec{z}))).
\] 
\end{itemize}
Another important consequence of \textsf{Collection} is that the levels of L\'{e}vy hierarchy are essentially closed under bounded quantification. In particular, \cite[Proposition 2.4]{flw16} notes that if $T$ is an $\mathcal{L}$-theory that proves $\mathsf{M}^-+\Pi_n\textsf{-Collection}$, then both $\Sigma_{n+1}^T$ and $\Pi_{n+1}^T$ are closed under bounded quantification.

We will use $\mathcal{M}, \mathcal{N}, \ldots$ to denote structures. If $\mathcal{M}$ is an $\mathcal{L}$-structure, then we will use $M$ to denote the underlying set of $\mathcal{M}$ and $\in^{\mathcal{M}}$ to denote the interpretation of $\in$ in $\mathcal{M}$, i.e. $\mathcal{M}= \langle M, \in^{\mathcal{M}} \rangle$. Let $\mathcal{M}= \langle M, \in^{\mathcal{M}} \rangle$ and $\mathcal{N}= \langle N, \in^{\mathcal{N}} \rangle$. We say that $\mathcal{M}$ is a {\bf substructure} of $\mathcal{N}$, and, abusing notation, write $\mathcal{M} \subseteq \mathcal{N}$, if $M \subseteq N$ and $\in^{\mathcal{N}}$ agrees with $\in^{\mathcal{M}}$ on $M \times M$. We say that $\mathcal{M}$ is an {\bf ($\Sigma_n$-)elementary substructure} of $\mathcal{N}$, and write $\mathcal{M} \prec \mathcal{N}$ ($\mathcal{M} \prec_n \mathcal{N}$, respectively) if $\mathcal{M} \subseteq \mathcal{N}$ and for all $\mathcal{L}$-formulae ($\Sigma_n$-formulae, respectively), $\phi(\vec{x})$, and for all all $\vec{a} \in M$, $\mathcal{M} \models \phi(\vec{a})$ if and only if $\mathcal{N} \models \phi(\vec{a})$. We say $\mathcal{N}$ is an {\bf end extension} of $\mathcal{M}$, and write $\mathcal{M} \subseteq_e \mathcal{N}$, if $\mathcal{M} \subseteq \mathcal{N}$ and for all $y \in M$ ad for all $x \in N$, if $\mathcal{N} \models x \in y$, then $x \in M$. The structure $\mathcal{N}$ is a {\bf powerset-preserving end extension} of $\mathcal{M}$, written $\mathcal{M} \subseteq_e^{\mathcal{P}} \mathcal{N}$, if $\mathcal{M} \subseteq_e \mathcal{N}$ and for all $y \in M$ and for  all $x \in N$, if $\mathcal{N} \models (x \subseteq y)$, then $x \in M$. We call $\mathcal{N}$ a {\bf $\Sigma_n$-elementary end extension} of $\mathcal{M}$, and write $\mathcal{M} \prec_{e, n} \mathcal{N}$, if $\mathcal{M} \subseteq_e \mathcal{N}$ and $\mathcal{M} \prec_n \mathcal{N}$. Note that $\Delta_0$-properties of points in an $\mathcal{L}$-structure $\mathcal{M}$ are preserved in every end extension, and $\Delta_0^{\mathcal{P}}$-properties of points in an $\mathcal{L}$-structure $\mathcal{M}$ are preserved in every powerset-preserving end extensions.

Both $\mathsf{KP}$ and $\mathsf{M}$ are capable of defining a ternary relation $\mathbin{\models}(u, m, x)$ (we write \mbox{$\langle u, \in \rangle \models \phi(x_1, \ldots, x_k)$} instead of $\mathbin{\models}(u, m, x)$, where $m = \ulcorner \phi(v_1, \ldots, v_k) \urcorner$ and $x= \langle x_1, \ldots, x_k \rangle$) that expresses that the formula with G\"{o}del code $m$ and parameters given by $x$ holds in the structure $\langle u, \in \rangle$. This formula is $\Delta_1^{\mathsf{KP}}$. For $\Gamma= \Sigma_n \textrm{ or } \Pi_n$, the theory $\mathsf{KP}$ defines a formula $\mathsf{Sat}_\Gamma(m, x)$ asserting that $m$ is the G\"{o}del code of a $\Gamma$-formula $\phi(v_1, \ldots, v_k)$, $x= \langle x_1, \ldots, x_k \rangle$ and $\phi(x_1, \ldots, x_k)$ holds. The formula $\mathsf{Sat}_{\Delta_0}(n, x)$ is $\Delta_1^{\mathsf{KP}}$, and if $\Gamma= \Sigma_n \textrm{ or } \Pi_n$ and $n \geq 1$, then $\mathsf{Sat}_\Gamma(m, x)$ is $\Gamma^{\mathsf{KP}}$. We use $X \prec_n \mathbb{V}$ to abbreviate the formula (with parameter $\omega$):
\[
(\forall m \in \omega)(\forall x \in X^{<\omega})(\mathsf{Sat}_{\Sigma_n}(m, x) \Rightarrow \langle X, \in \rangle \models \mathsf{Sat}_{\Sigma_n}(m, x)),
\]     
asserting that $\langle X, \in \rangle$ is substructure of the universe that preserves $\Sigma_n$-properties of tuples in $X$. The formula $X \prec_n \mathbb{V}$ is $\Pi_n^{\mathsf{KPI}}$.

We use $\mathsf{Ord}$ to denote to class of ordinals. Verifying that the treatment in \cite[Chapter II]{bar75} appeals to no more than $\Pi_1\textsf{-Foundation}$ reveals that the theory $\mathsf{KP}$ is capable of defining G\"{o}del's constructible universe:
\[
L_0= \emptyset, L_\alpha= \bigcup_{\beta \in \alpha} L_\beta \textrm{ if } \alpha \textrm{ is a limit ordinal, and}
\]
\[
L_{\alpha+1}= L_\alpha \cup \mathsf{Def}(L_\alpha), \textrm{ and } L= \bigcup_{\alpha \in \mathsf{Ord}} L_\alpha,
\]
where $\mathsf{Def}(X)$ is the set of subsets of $X$ that are the extension of a first-order formulae (with parameter) according to the structure $\langle X, \in \rangle$. In particular, the function $\alpha \mapsto L_\alpha$ is provably total in $\mathsf{KP}$, and $\Delta_1^{\mathsf{KP}}$. The following axiom asserts that every set is member of the constructible universe and can be expressed as a $\Pi_2$-sentence in $\mathsf{KP}$:
\begin{itemize}
\item[]($\mathsf{V=L}$) $\forall x \exists \alpha (x \in L_\alpha)$.
\end{itemize}

\begin{definition}
Let $T$ be an $\mathcal{L}$-theory. A transitive set $M$ is the {\bf minimum model} of $T$ if $\langle M, \in \rangle \models T$ and for all transitive $N$ such that $\langle N, \in \rangle \models T$, $M \subseteq N$. 
\end{definition}

Gostanian \cite{gos80} shows that many natural subsystems of $\mathsf{ZF}$ have minimal models. In particular, the following is a consequence of \cite[Theorem 1.6(b)]{gos80}:

\begin{theorem} \label{Th:ExistenceOfMinimumModels}
Let $n \in \omega$ with $n \geq 1$. The theory $\mathsf{M}+\Pi_n\textsf{-Collection}$ has a minimum model. Moreover, the minimum model of $\mathsf{M}+\Pi_n\textsf{-Collection}$ is $L_\alpha$, where $\alpha \in \omega_1$. \Square 
\end{theorem}

In the presence of enough \textsf{Foundation}, the theory $\mathsf{M}+\Pi_{n+1}\textsf{Collection}$ is capable of building a transitive model of $\mathsf{M}+\Pi_n\textsf{-Collection}$. Theorem \ref{Th:RelativeConsistencyOfCollection} is a weakening of \cite[Theorem 4.4]{mck19} obtained by observing that $\Pi_{n}\textsf{-Foundation}$ implies that induction on the natural numbers holds for all $\Sigma_{n}$-formulae. 

\begin{theorem} \label{Th:RelativeConsistencyOfCollection}
Let $n \in \omega$ with $n \geq 1$. The theory $\mathsf{M}+\Pi_{n+1}\textsf{-Collection}+\Pi_{n+2}\textsf{-Foundation}$ proves that there exists a transitive model of $\mathsf{M}+\mathsf{Separation}+\Pi_n\textsf{-Collection}$.
\end{theorem}

Theorems \ref{Th:ExistenceOfMinimumModels} and \ref{Th:RelativeConsistencyOfCollection} immediately imply that the minimum models of $\mathsf{M}+\Pi_n\textsf{-Collection}$ and $\mathsf{M}+\Pi_{n+1}\textsf{-Collection}$ do not coincide.

\begin{corollary} \label{Th:SeparationOfMinimumModels}
Let $n \in \omega$ with $n \geq 1$. Let $\alpha \in \omega_1$ be such that $L_\alpha$ is the minimum model of $\mathsf{M}+\Pi_n\textsf{-Collection}$. Then $\langle L_\alpha, \in \rangle$ does not satisfy $\Pi_{n+1}\textsf{-Collection}$. \Square
\end{corollary}

\cite[\S 4]{mck19} shows that, for $n \geq 1$, $\mathsf{M}+\Pi_n\textsf{-Collection}+\Pi_{n+1}\textsf{-Foundation}$ proves the consistency of $\mathsf{M}+\Pi_n\textsf{-Collection}$. Combined with the fact that, for all $n \geq 1$, $\mathsf{M}+\Pi_n\textsf{-Collection}$ proves that $\mathsf{M}+\Pi_n\textsf{-Collection}$ holds in its own constructible universe, this yields: 

\begin{theorem} \label{Th:CollectionDoesNotProveFoundation}
Let $n \in \omega$ with $n \geq 1$. The theory $\mathsf{M}+\Pi_n\textsf{-Collection}+\mathsf{V=L}$ does not prove $\Pi_{n+1}\textsf{-Foundation}$. \Square  
\end{theorem}

The set theory $\mathsf{MOST}$, studied in \cite{mat01}, is obtained from $\mathsf{M}$ by adding $\Delta_0\textsf{-Collection}$, $\Sigma_1\textsf{-Separation}$ and the Axiom of Choice ($\mathsf{AC}$). In particular, $\mathsf{MOST}$ is a subtheory of $\mathsf{M}+\Pi_1\textsf{-Collection}+\mathsf{V=L}$. A key feature of $\mathsf{MOST}$, and the source of its name, is the fact that it proves {\it Mostowski's Collapsing Lemma}: Every well-founded extensional relation is isomorphic to a unique transitive set by a unique isomorphism known as the {\bf transitive collapse}. This, combined with the absoluteness of the levels of the constructible universe between transitive models and the first-order expressibility of being constructible, allows us to apply {\it G\"{o}del's Condensation Lemma} in $\mathsf{MOST}$:

\begin{lemma}
The theory $\mathsf{MOST}$ proves that for all limit ordinals $\delta$, if $\langle X, \in \rangle \prec \langle L_\delta, \in \rangle$, then there exists $\alpha \leq \delta$ such that the transitive collapse of $\langle X, \in \rangle$ is isomorphic to $\langle L_\alpha, \in \rangle$. \Square  
\end{lemma}  

A key ingredient used in the results proved in sections \ref{Sec:BaseCase} and \ref{Sec:GeneralCase} is a generalisation of a result of Simpson (see \cite[Remark 2]{kau81}) that is proved in \cite[Theorem 3.1]{mck25}:

\begin{theorem} \label{Th:SimpsonResult}
Let $n \in \omega$ with $n \geq 1$. Let $\mathcal{M}= \langle M, \in^{\mathcal{M}} \rangle$ be such that $\mathcal{M} \models \mathsf{KP}+\mathsf{V=L}$. Suppose that $\mathcal{N}= \langle N, \in^{\mathcal{N}} \rangle$ is such that $\mathcal{M} \prec_{e, n} \mathcal{N}$, $\mathcal{N} \models \mathsf{KP}$ and $\mathsf{Ord}^{\mathcal{N}} \backslash \mathsf{Ord}^{\mathcal{M}}$ is nonempty and has no least element. If $\mathcal{N} \models \Pi_{n-1}\textsf{-Collection}$ or $\mathcal{N} \models \Pi_{n+2}\textsf{-Foundation}$, then $\mathcal{M} \models \Pi_n\textsf{-Collection}$. \Square 
\end{theorem}

Theorem \ref{Th:SimpsonResult} is used in \cite[\S 3]{mck25} to show that for $n \geq 1$, the minimum model of $\mathsf{M}+\mathsf{Separation}+\Pi_n\textsf{-Collection}$ has no proper $\Sigma_{n+1}$-elementary end extension that satisfies either $\mathsf{KP}+\Pi_n\textsf{-Collection}$ or $\mathsf{KP}+\Pi_{n+3}\textsf{-Foundation}$.  

\section[The theory transferred to partially-elementary end extensions]{The theory transferred to partially-elementary end extensions} \label{Sec:ComplexityCalculations}

In this section we show that for all $n \geq 2$, if $\mathcal{N}$ is a $\Sigma_{n+1}$-elementary end extension of $\mathcal{M}$ that satisfies $\mathsf{M}^-+\Pi_n\textsf{-Collection}+\mathsf{V=L}$, then $\mathcal{N}$ satisfies $\mathsf{M}^-+\Pi_{n-2}\textsf{-Collection}+\Sigma_{n-1}\textsf{-Separation}+\mathsf{V=L}$. Moreover, if in addition $\mathcal{M}$ satisfies \textsf{Powerset} and $\mathcal{N}$ satisfies $\Pi_{n-1}\textsf{-Collection}$, then $\mathcal{N}$ satisfies $\Sigma_n\textsf{-Separation}$. We also show that if $\mathcal{N}$ is a $\Sigma_2$-elementary end extension of $\mathcal{M}$ that satisfies $\mathsf{M}^-+\Pi_1\textsf{-Collection}$ and $\mathcal{N}$ satisfies $\Delta_0\textsf{-Collection}$, then $\mathcal{N}$ satisfies $\Sigma_1 \cup \Pi_1\textsf{-Foundation}$.

\begin{lemma} \label{Th:ComplexityOfCollection}
Let $n \in \omega$. The theory $\mathsf{M}^-+\Pi_n\textsf{-Collection}$ is axiomatised by a collection of $\Pi_{n+3}$-sentences.
\end{lemma}

\begin{proof}
We prove this by induction on $n$. We have already observed that $\mathsf{M}^-$ is axiomatised by a collection of $\Pi_2$-sentences.    Every instance of $\Delta_0\textsf{-Collection}$ is of the form
\[
\forall \vec{z} \forall w((\forall x \in w) \exists y \phi(x, y, \vec{z}) \Rightarrow \exists c(\forall x \in w)(\exists y \in c) \phi(x, y, \vec{z})),
\] 
where $\phi(x, y, \vec{z})$ is $\Delta_0$. Each of these instances is equivalent to a $\Pi_3$-sentence. This shows that the lemma holds when $n=0$. Now, let $k \geq 0$ and assume that $\Gamma$ is a collection of $\Pi_{k+3}$-sentences that axiomatise $\mathsf{M}^-+\Pi_k\textsf{-Collection}$. Let $\phi(x, y, \vec{z})$ be a $\Pi_{k+1}$-formula. The instance of $\Pi_{k+1}\textsf{-Collection}$ for the formula $\phi(x, y, \vec{z})$ is
\begin{equation} \label{eq:CollectionInstance}
\forall \vec{z} \forall w((\forall x \in w) \exists y \phi(x, y, \vec{z}) \Rightarrow \exists c(\forall x \in w)(\exists y \in c) \phi(x, y, \vec{z})).
\end{equation}
In the theory $\mathsf{M}^-+\Pi_k\textsf{-Collection}$, (\ref{eq:CollectionInstance}) is equivalent to a $\Pi_{k+4}$-sentence. Therefore, there is a collection, $\Lambda$, of $\Pi_{k+4}$ sentence such that $\Gamma \cup \Lambda$ axiomatises $\mathsf{M}^-+\Pi_{k+1}\textsf{-Collection}$. \Square
\end{proof}

\begin{lemma} \label{Th:ComplexityOfSeparation}
Let $n \in \omega$. For each instance, $\sigma$, of $\Sigma_{n+1}\textsf{-Separation}$, there exists a $\Pi_{n+3}$-sentence $\gamma$ such that $\mathsf{M}^-+\Pi_n\textsf{-Collection} \vdash (\sigma \iff \gamma)$.
\end{lemma}

\begin{proof}
Let $\sigma$ be an instance of $\Sigma_{n+1}\textsf{-Separation}$. Therefore, $\sigma$ is
\[
\forall \vec{z} \forall w \exists y ((\forall x \in w)(\phi(x, \vec{z}) \Rightarrow x \in y) \land (\forall x \in y)(x \in w \land \phi(x, \vec{z}))),
\]
where $\phi(x, \vec{z})$ is a $\Sigma_{n+1}$-formulae. In the theory $\mathsf{M}^-+\Pi_n\textsf{-Collection}$, this sentence is equivalent to a $\Pi_{n+3}$-sentence.
\Square
\end{proof}

\begin{lemma} \label{Th:ComplexityOfFoundation}
Let $n \in \omega$. There exists a collection, $\Gamma$, of $\Pi_{n+2}$-sentences such that
\[
\mathsf{M}^-+\Pi_n\textsf{-Collection}+\Gamma \vdash \Pi_{n+1} \cup \Sigma_{n+1}\textsf{-Foundation} 
\]
\[
\textrm{and } \mathsf{M}^-+\Pi_n\textsf{-Collection}+\Pi_{n+1} \cup \Sigma_{n+1}\textsf{-Foundation} \vdash \Gamma. 
\]
\end{lemma}

\begin{proof}
$\Sigma_{n+1} \cup \Pi_{n+1}\textsf{-Foundation}$ is the scheme: for all $\Sigma_{n+1} \cup \Pi_{n+1}$-formulae, $\phi(x, \vec{z})$,
\[
\forall \vec{z} (\exists x \phi(x, \vec{z}) \Rightarrow \exists y(\phi(y, \vec{z}) \land (\forall w \in y)(\neg \phi(w, \vec{z})))).
\]
In the theory $\mathsf{M}^-$, $\Sigma_{n+1} \cup \Pi_{n+1}\textsf{-Foundation}$ is equivalent to the scheme: for all $\Sigma_{n+1} \cup \Pi_{n+1}$-formulae, $\phi(x, \vec{z})$,
\begin{equation}\label{eq:SimplifiedFoundationScheme}
\forall \vec{z} \forall u ((\exists x \in u) \phi(x, \vec{z}) \Rightarrow (\exists y \in u)(\phi(y, \vec{z}) \land (\forall w \in y)(\neg \phi(w, \vec{z})))).
\end{equation}
In the theory $\mathsf{M}^-+\Pi_n\textsf{-Collection}$, each instance (\ref{eq:SimplifiedFoundationScheme}) is equivalent to a $\Pi_{n+2}$-sentence.
\Square
\end{proof}

\begin{theorem} \label{Th:TheoryInEndExtensionWithoutCollection}
Let $n \in \omega$ with $n \geq 2$. Let $\mathcal{M}= \langle M, \in^{\mathcal{M}} \rangle$ be such that $\mathcal{M} \models \mathsf{M}+\Pi_n\textsf{-Collection}+\mathsf{V=L}$. If $\mathcal{N}= \langle N, \in^{\mathcal{N}} \rangle$ is such that $\mathcal{M} \prec_{e, n+1} \mathcal{N}$, then 
\[
\mathcal{N} \models \mathsf{M}+\Pi_{n-2}\textsf{-Collection}+\Sigma_{n-1}\textsf{-Separation}+\mathsf{V=L}.
\]
\end{theorem}

\begin{proof}
Let $\mathcal{N}= \langle N, \in^{\mathcal{N}} \rangle$ be such that $\mathcal{M} \prec_{e, n+1} \mathcal{N}$. Since $n \geq 2$ and $\mathcal{M} \prec_{e, n+1} \mathcal{N}$, $\mathcal{N} \models \mathsf{M}+\mathsf{V=L}$. So, by Lemma \ref{Th:ComplexityOfCollection}, $\mathcal{N} \models \mathsf{M}+\Pi_{n-2}\textsf{-Collection}+\mathsf{V=L}$. Therefore, by Lemma \ref{Th:ComplexityOfSeparation}, $\mathcal{N} \models \mathsf{M}+\Pi_{n-2}\textsf{-Collection}+\Sigma_{n-1}\textsf{-Separation}+\mathsf{V=L}$.
\Square
\end{proof}

\begin{theorem} \label{Th:KPIInSigma2EndExtension}
Let $\mathcal{M}= \langle M, \in^{\mathcal{M}} \rangle$ be such that $\mathcal{M} \models \mathsf{M}+\Pi_1\textsf{-Collection}+\mathsf{V=L}$. If $\mathcal{N}= \langle N, \in^{\mathcal{N}} \rangle$ is such that $\mathcal{M} \prec_{e, 2} \mathcal{N}$ and $\mathcal{N} \models \Delta_0\textsf{-Collection}$, then 
\[
\mathcal{N} \models \mathsf{KPI}+\Sigma_1\textsf{-Foundation}+\mathsf{V=L}.
\]
\end{theorem}

\begin{proof}
Let $\mathcal{N}= \langle N, \in^{\mathcal{N}} \rangle$ be such that $\mathcal{M} \prec_{e, 2} \mathcal{N}$ and $\mathcal{N} \models \Delta_0\textsf{-Collection}$. Since $\mathcal{M} \prec_{e, 2} \mathcal{N}$, $\mathcal{N} \models \mathsf{M}^-+\Delta_0\textsf{-Collection}+\mathsf{V=L}$. Therefore, by Lemma \ref{Th:ComplexityOfFoundation},
\[
\mathcal{N} \models \mathsf{KPI}+\Sigma_1\textsf{-Foundation}+\mathsf{V=L}.
\] 
\Square
\end{proof}

The presence \textsf{Powerset} in $\mathsf{M}$ allows $\Sigma_{n+1}\textsf{-Separation}$ to be derived from $\Pi_n\textsf{-Collection}$ together with $\Pi_n\textsf{-Foundation}$. The following is \cite[Theorem 5.9]{mck25}:

\begin{theorem} \label{Th:FoundationImpliesSeparation}
Let $n \in \omega$ with $n \geq 1$. The theory $\mathsf{M}+\Pi_n\textsf{-Collection}+\Pi_{n+1}\textsf{-Foundation}$ proves $\Sigma_{n+1}\textsf{-Separation}$. \Square
\end{theorem} 

This means that if we have enough partial-elementarity to transfer \textsf{Powerset} to the end extension and we know that the end extension satisfies enough collection, then more separation is also transferred to the end extension. 

\begin{theorem} \label{Th:TheoryInEndExtensionWithCollection}
Let $n \in \omega$ with $n \geq 2$. Let $\mathcal{M}= \langle M, \in^{\mathcal{M}} \rangle$ be such that $\mathcal{M} \models \mathsf{M}+\Pi_n\textsf{-Collection}+\mathsf{V=L}$. If $\mathcal{N}= \langle N, \in^{\mathcal{N}} \rangle$ is such that $\mathcal{M} \prec_{e, n+1} \mathcal{N}$ and $\mathcal{N} \models \Pi_{n-1}\textsf{-Collection}$, then 
\[
\mathcal{N} \models \mathsf{M}+\Pi_{n-1}\textsf{-Collection}+\Sigma_n\textsf{-Separation}+\mathsf{V=L}. 
\]
\end{theorem}

\begin{proof}
Let $\mathcal{N}= \langle N, \in^{\mathcal{N}} \rangle$ be such that $\mathcal{M} \prec_{e, n+1} \mathcal{N}$ and $\mathcal{N} \models \Pi_{n-1}\textsf{-Collection}$. Since $n \geq 2$ and $\mathcal{M} \prec_{e, n+1} \mathcal{N}$, $\mathcal{N} \models \mathsf{M}+\Pi_{n-1}\textsf{-Collection}+\mathsf{V=L}$. So, by Lemma \ref{Th:ComplexityOfFoundation}, $\mathcal{N} \models \mathsf{M}+\Pi_{n-1}\textsf{-Collection}+\Pi_{n}\textsf{-Foundation}+\mathsf{V=L}$. Therefore, since $n \geq 2$, Theorem \ref{Th:FoundationImpliesSeparation} yields $\mathcal{N} \models \mathsf{M}+\Pi_{n-1}\textsf{-Collection}+\Sigma_{n}\textsf{-Separation}+\mathsf{V=L}$.
\Square
\end{proof}

\section[$\Sigma_2$-elementary end extensions]{$\Sigma_2$-elementary end extensions that satisfy $\Delta_0$-Collection} \label{Sec:BaseCase}

In this section we show that if $\mathcal{M}$ is a model of $\mathsf{M}+\Pi_1\textsf{-Collection}+\mathsf{V=L}$ that has a $\Sigma_2$-elementary end extension, $\mathcal{N}$, that satisfies $\Delta_0\textsf{-Collection}$ and such that $\mathcal{N}$ contains a new ordinal but no least new ordinal, then $\mathcal{M}$ musty satisfy $\Pi_2\textsf{-Collection}$. This result is used to provide a negative answer to Kaufmann's Question (Question \ref{Q:KaufmannQuestion}) by showing that the minimum model of $\mathsf{M}+\Pi_1\textsf{-Collection}$ has no proper $\Sigma_2$-elementary end extension that satisfies $\Delta_0\textsf{-Collection}$.

The following is a slight strengthening of \cite[Lemma 5.7]{mck25} that is obtained by observing that the proof of \cite[Lemma 5.7]{mck25} makes no use of the assumption that the end extension satisfies $\mathsf{M}$:

\begin{lemma} \label{Th:ElementarityImpliesPowersetPreserving}
Let $\mathcal{M}= \langle M, \in \rangle$ and $\mathcal{N}= \langle N, \in^{\mathcal{N}} \rangle$ be such that $\mathcal{M} \models \mathsf{M}$. If $\mathcal{M} \prec_{e, 1} \mathcal{N}$, then $\mathcal{M} \subseteq_e^{\mathcal{P}} \mathcal{N}$. \Square
\end{lemma}

Note that a proper $\Sigma_2$-elementary end extension of a model of $\mathsf{M}+\Pi_1\textsf{-Collection}$ that satisfies $\Delta_0\textsf{-Collection}$ need not satisfy \textsf{Powerset} or the assertion that there is no greatest cardinal. Nevertheless, an overspill argument allows us to see that such an end extension must contain a new cardinal.

\begin{lemma} \label{Th:NewCardinalInExtension}
Let $\mathcal{M}= \langle M, \in^{\mathcal{M}} \rangle$ be such that $\mathcal{M} \models \mathsf{M}+\Pi_1\textsf{-Collection}+\mathsf{V=L}$. If $\mathcal{N}= \langle N, \in^{\mathcal{N}} \rangle$ is such that $\mathcal{M} \prec_{e, 2} \mathcal{N}$, $M \neq N$ and $\mathcal{N} \models \Delta_0\textsf{-Collection}$, then there exists $\kappa \in \mathsf{Ord}^{\mathcal{N}} \backslash \mathsf{Ord}^{\mathcal{M}}$ such that $\mathcal{N} \models (\kappa \textrm{ is a cardinal})$. 
\end{lemma}

\begin{proof}
Let $\mathcal{N}= \langle N, \in^{\mathcal{N}} \rangle$ be such that $\mathcal{M} \prec_{e, 2} \mathcal{N}$, $M \neq N$ and $\mathcal{N} \models \Delta_0\textsf{-Collection}$. By Theorem \ref{Th:KPIInSigma2EndExtension},
\[
\mathcal{N} \models \mathsf{KPI}+\Sigma_1\textsf{-Foundation}+\mathsf{V=L}.
\]
And, by Lemma \ref{Th:ElementarityImpliesPowersetPreserving}, $\mathcal{M} \subseteq_e^{\mathcal{P}} \mathcal{N}$. Using the fact that $M \neq N$, let $\beta \in \mathsf{Ord}^{\mathcal{N}} \backslash \mathsf{Ord}^{\mathcal{M}}$. Using $\Sigma_1\textsf{-Foudation}$ in $\mathcal{N}$, let $\kappa$ be the $\in$-least element of the class
\[
A= \{ \gamma \in \mathsf{Ord}^{\mathcal{N}} \mid \exists f(f: \gamma \longrightarrow \beta \textrm{ is a bijection})\}.
\]
Then $\mathcal{N} \models (\kappa \textrm{ is a cardinal})$. Now, suppose that $\kappa \in M$. But then, since $\mathcal{M} \subseteq_e^{\mathcal{P}} \mathcal{N}$, the well-ordering of $\kappa$ with order-type $\beta$ is also in $M$. Since $\mathcal{M} \models \mathsf{M}+\Pi_1\textsf{-Collection}+\mathsf{V=L}$, this implies that $\beta \in M$, which is a contradiction. Therefore, $\kappa \in \mathsf{Ord}^{\mathcal{N}} \backslash \mathsf{Ord}^{\mathcal{M}}$. \Square     
\end{proof}

In the theory $\mathsf{ZF}+\mathsf{V=L}$, if $\kappa > \omega$ is cardinal, then $\langle L_\kappa, \in \rangle$ satisfies $\mathsf{KPI}+\textsf{Separation}$, the universe is a powerset-preserving end extension of $L_\kappa$ and the universe is a $\Sigma_1$-elementary end extension of $L_\kappa$. Here we verify that these properties of $L_\kappa$, where $\kappa > \omega$ is a cardinal, can be proved in the theory $\mathsf{M}+\Pi_1\textsf{-Collection}+\mathsf{V=L}$.    

\begin{lemma} \label{Th:LKappaPowersetPreserving}
The theory $\mathsf{M}+\Pi_1\textsf{-Collection}+\mathsf{V=L}$ proves that if $\kappa > \omega$ is a cardinal, $x \in L_\kappa$ and $y \subseteq x$, then $y \in L_\kappa$. 
\end{lemma}

\begin{proof}
Work in the theory $\mathsf{M}+\Pi_1\textsf{-Collection}+\mathsf{V=L}$. Let $\kappa > \omega$ be a cardinal. Let $x \in L_\kappa$ and let $y \subseteq x$. Let $\gamma \in \kappa$ be such that $x \in L_\gamma$. Note that $|L_\gamma|< \kappa$. Let $\lambda$ be a limit ordinal such that $y, L_\gamma \in L_\lambda$. Taking a Skolem Hull of $\langle L_\lambda, \in \rangle$, let $\langle X, \in \rangle \prec \langle L_\lambda, \in \rangle$ be such that $L_\gamma \cup \{y\} \subseteq X$ and $|X| < \kappa$. Taking the transitive collapse of $\langle X, \in \rangle$ and using G\"{o}del's Condensation Lemma, we obtain $L_\beta \in L_\kappa$ such that $y \in L_\beta$. Therefore $y \in L_\kappa$. \Square
\end{proof}

\begin{lemma} \label{Th:LKappaAdmissible}
The theory $\mathsf{M}+\Pi_1\textsf{-Collection}+\mathsf{V=L}$ proves that if $\kappa > \omega$ is a cardinal, then $\langle L_\kappa, \in \rangle \models \mathsf{KPI}+\textsf{Separation}+\mathsf{V=L}$.  
\end{lemma}

\begin{proof}
Work in the theory $\mathsf{M}+\Pi_1\textsf{-Collection}+\mathsf{V=L}$. Let $\kappa > \omega$ is a cardinal. It follows from Lemma \ref{Th:LKappaPowersetPreserving} and the fact that $L_\kappa$ is a set that $\langle L_\kappa, \in \rangle \models \mathsf{M}^-+\textsf{Separation}$. To verify $\Delta_0\textsf{-Collection}$, let $\phi(x, y, \vec{z})$ be a $\Delta_0$-formula. Let $\vec{a}, b \in L_\kappa$ be such that
\[
\langle L_\kappa, \in \rangle \models (\forall x \in b) \exists y \phi(x, y, \vec{a}).
\]
Let $\gamma \in \kappa$ be such that $\vec{a}, b \in L_\gamma$. Note that $|L_\gamma| < \kappa$. Let $\lambda > \kappa$ be a limit ordinal. Since $\phi(x, y, \vec{z})$ is $\Delta_0$ and $L_\kappa \in L_\lambda$,
\[
\langle L_\lambda, \in \rangle \models \exists c (\forall x \in b)(\exists y \in c) \phi(x, y, \vec{a}).
\]
Taking the Skolem Hull of $\langle L_\lambda, \in \rangle$, let $\langle X, \in \rangle \prec \langle L_\lambda, \in \rangle$ be such that $L_\gamma \subseteq X$ and $|X|< \kappa$. Taking the transitive collapse of $\langle X, \in \rangle$, we obtain $L_\beta \in L_\kappa$ and an isomorphism between $\langle X, \in \rangle$ and $\langle L_\beta, \in \rangle$ that is the identity on $\vec{a}, b$. Therefore,
\[
\langle L_\beta, \in \rangle \models \exists c (\forall x \in b)(\exists y \in c) \phi(x, y, \vec{a}).
\]
And so, since $\phi(x, y, \vec{a})$ is $\Delta_0$,
\[
\langle L_\kappa, \in \rangle \models \exists c (\forall x \in b)(\exists y \in c) \phi(x, y, \vec{a}).
\]
This shows that $\langle L_\kappa, \in \rangle$ satisfies $\Delta_0\textsf{-Collection}$, and so
\[
\langle L_\kappa, \in \rangle \models \mathsf{KPI}+\textsf{Separation}+\mathsf{V=L}.
\]   
\Square
\end{proof}

\begin{lemma} \label{Th:LKappaSigma1Elementary}
The theory $\mathsf{M}+\Pi_1\textsf{-Collection}+\mathsf{V=L}$ proves that if $\kappa > \omega$ is a a cardinal, then for all $k \in \omega$ and for all $x \in L_\kappa$, if $\mathsf{Sat}_{\Sigma_1}(k, x)$, then $\langle L_\kappa, \in \rangle \models \mathsf{Sat}_{\Sigma_1}(k, x)$.
\end{lemma}

\begin{proof}
Work in the theory $\mathsf{M}+\Pi_1\textsf{-Collection}+\mathsf{V=L}$. Let $\kappa > \omega$ be a cardinal. Let $\phi(k, x, y)$ be a $\Delta_0$-formula such that 
\[
\mathsf{KPI} \vdash \forall k \forall x(\mathsf{Sat}_{\Sigma_1}(k, x) \iff \exists y \phi(k, x, y)).  
\]
Let $k \in \omega$ and let $x \in L_\kappa$. Suppose that $\mathsf{Sat}_{\Sigma_1}(k, x)$ holds. Let $y$ be such that $\phi(k, x, y)$ holds. Let $\lambda> \kappa$ be a limit ordinal such that $y \in  L_\lambda$. So, $\langle L_\lambda, \in \rangle \models \exists y \phi(k, x, y)$. Let $\gamma \in \kappa$ with be such that $\omega, x \in L_\gamma$. Note that $|L_\gamma|< \kappa$. Taking the Skolem Hull of $\langle L_\lambda, \in \rangle$, let $\langle X, \in \rangle \prec \langle L_\lambda, \in \rangle$ be such that $L_\gamma \subseteq X$ and $|X| < \kappa$. Taking the collapse $\langle X, \in \rangle$, we obtain $L_\beta \in L_\kappa$ and an isomorphism between $\langle X, \in \rangle$ and $\langle L_\beta, \in \rangle$ that is the identity on $x$ and $k$. Therefore, $\langle L_\beta, \in \rangle \models \exists y \phi(k, x, y)$. So, since $\phi$ is $\Delta_0$, $\langle L_\kappa, \in \rangle \models \exists y \phi(k, x, y)$. Therefore, $\langle L_\kappa, \in \rangle \models \mathsf{Sat}_{\Sigma_1}(k, x)$. 
\Square 
\end{proof}

The proofs of Lemmas \ref{Th:LKappaPowersetPreserving}, \ref{Th:LKappaAdmissible} and \ref{Th:LKappaSigma1Elementary} make  use of the fact that $\mathsf{M}+\Pi_1\textsf{-Collection}+\mathsf{V=L}$ proves the Mostowski Collapsing Lemma, a result that is not provable in $\mathsf{KPI}+\Sigma_1\textsf{-Foundation}+\mathsf{V=L}$. Despite this, if $\mathcal{M}$ satisfies $\mathsf{M}+\Pi_1\textsf{-Collection}+\mathsf{V=L}$, then we can exploit elementarity to show that if $\mathcal{M} \prec_{e, 2} \mathcal{N}$ and $\mathcal{N}$ satisfies $\Delta_0\textsf{-Collection}$, then the properties of $L_\kappa$ obtained from Lemmas \ref{Th:LKappaPowersetPreserving}, \ref{Th:LKappaAdmissible} and \ref{Th:LKappaSigma1Elementary} also hold for $L_\kappa$ in $\mathcal{N}$ where $\kappa$ is a cardinal in $\mathsf{Ord}^{\mathcal{N}} \backslash \mathsf{Ord}^{\mathcal{M}}$.

\begin{lemma} \label{Th:LKappaAdmissibleInExtension}
Let $\mathcal{M}= \langle M, \in^{\mathcal{M}} \rangle$ be such that $\mathcal{M} \models \mathsf{M}+\Pi_1\textsf{-Collection}+\mathsf{V=L}$. Let $\mathcal{N}= \langle N, \in^{\mathcal{N}} \rangle$ be such that $\mathcal{M} \prec_{e, 2} \mathcal{N}$ and $\mathcal{N} \models \Delta_0\textsf{-Collection}$. If $\kappa \in \mathsf{Ord}^{\mathcal{N}} \backslash \mathsf{Ord}^{\mathcal{M}}$ is such that $\mathcal{N} \models (\kappa \textrm{ is a cardinal})$, then
\[
\langle (L_\kappa^{\mathcal{N}})^*, \in^{\mathcal{N}} \rangle \models \mathsf{KPI}+\mathsf{Separation}+\mathsf{V=L}.
\] 
\end{lemma}

\begin{proof}
By Theorem \ref{Th:KPIInSigma2EndExtension}, $\mathcal{N} \models \mathsf{KPI}+\Sigma_1\textsf{-Foundation}+\mathsf{V=L}$. Let $\kappa \in \mathsf{Ord}^{\mathcal{N}} \backslash \mathsf{Ord}^{\mathcal{M}}$ be such that $\mathcal{N} \models (\kappa \textrm{ is a cardinal})$. Let $\sigma_0(z)$ be the formula
\[
\forall \lambda((z \in \lambda \land \lambda \textrm{ is a cardinal}) \Rightarrow (\langle L_\lambda, \in \rangle \models \mathsf{KPI}+\mathsf{Separation}+\mathsf{V=L})).
\]
Note that $\sigma_0(z)$ is $\Pi_2^{\mathsf{KPI}}$. Now, Lemma \ref{Th:LKappaAdmissible} shows that $\mathcal{M} \models \sigma_0(\omega)$. So, since $\mathcal{M} \prec_{e, 2} \mathcal{N}$ and $\mathcal{N} \models \mathsf{KPI}$, $\mathcal{N} \models \sigma_0(\omega)$. Therefore, 
\[
\langle (L_\kappa^{\mathcal{N}})^*, \in^{\mathcal{N}} \rangle \models \mathsf{KPI}+\mathsf{Separation}+\mathsf{V=L}.
\]
\Square
\end{proof}

\begin{lemma} \label{Th:ElementarityOfLKappaInExtension}
Let $\mathcal{M}= \langle M, \in^{\mathcal{M}} \rangle$ be such that $\mathcal{M} \models \mathsf{M}+\Pi_1\textsf{-Collection}+\mathsf{V=L}$. Let $\mathcal{N}= \langle N, \in^{\mathcal{N}} \rangle$ be such that $\mathcal{M} \prec_{e, 2} \mathcal{N}$ and $\mathcal{N} \models \Delta_0\textsf{-Collection}$. If $\kappa \in \mathsf{Ord}^{\mathcal{N}} \backslash \mathsf{Ord}^{\mathcal{M}}$ is such that $\mathcal{N} \models (\kappa \textrm{ is a cardinal})$, then $\mathcal{M} \prec_{e, 2} \langle (L_\kappa^{\mathcal{N}})^*, \in^{\mathcal{N}} \rangle$.
\end{lemma}

\begin{proof}
By Theorem \ref{Th:KPIInSigma2EndExtension}, $\mathcal{N} \models \mathsf{KPI}+\Sigma_1\textsf{-Foundation}+\mathsf{V=L}$. Let $\kappa \in \mathsf{Ord}^{\mathcal{N}} \backslash \mathsf{Ord}^{\mathcal{M}}$ be such that $\mathcal{N} \models (\kappa \textrm{ is a cardinal})$. Let $\sigma_1(z)$ be the formula
\[
\forall \lambda \left(\begin{array}{c}
(z \in \lambda \land \lambda \textrm{ is a cardinal}) \Rightarrow\\
(\forall k \in \omega)(\forall x \in L_\lambda)(\mathsf{Sat}_{\Sigma_1}(k, x) \Rightarrow \langle L_\lambda, \in \rangle \models \mathsf{Sat}_{\Sigma_1}(k, x))
\end{array} \right).
\]
Note that $\sigma_1(z)$ is $\Pi_2^{\mathsf{KPI}}$. Lemma \ref{Th:LKappaSigma1Elementary} shows that $\mathcal{M} \models \sigma_1(\omega)$. So, since $\mathcal{M} \prec_{e, 2} \mathcal{N}$ and $\mathcal{N} \models \mathsf{KPI}$, $\mathcal{N} \models \sigma_1(\omega)$. In particular,
\[
\mathcal{M} \prec_{e, 1} \langle (L_\kappa^{\mathcal{N}})^*, \in^{\mathcal{N}} \rangle \prec_{e, 1} \mathcal{N}.
\]
Now, let $\phi(x, \vec{z})$ be a $\Pi_1$-formula. Let $\vec{a} \in M$ and suppose that $\langle (L_\kappa^{\mathcal{N}})^*, \in^{\mathcal{N}} \rangle \models \exists x \phi(x, \vec{a})$. Let $b \in (L_\kappa^{\mathcal{N}})^*$ be such that $\langle (L_\kappa^{\mathcal{N}})^*, \in^{\mathcal{N}} \rangle \models \phi(b, \vec{a})$. Since $\langle (L_\kappa^{\mathcal{N}})^*, \in^{\mathcal{N}} \rangle \prec_{e, 1} \mathcal{N}$, $\mathcal{N} \models \phi(b, \vec{a})$. In particular, $\mathcal{N} \models \exists x \phi(x, \vec{a})$. Therefore, since $\mathcal{M} \prec_{e, 2} \mathcal{N}$, $\mathcal{M} \models \exists x \phi(x, \vec{a})$. This shows that $\mathcal{M} \prec_{e, 2} \langle (L_\kappa^{\mathcal{N}})^*, \in^{\mathcal{N}} \rangle$.
\Square
\end{proof}

This allows us to prove the base case ($n=1$) of the Main Theorem (Theorem \ref{Th:MainTheorem}).

\begin{theorem} \label{Th:BaseCaseMainTheorem}
Let $\mathcal{M}= \langle M, \in^{\mathcal{M}} \rangle$ be such that $\mathcal{M} \models \mathsf{M}+\Pi_1\textsf{-Collection}+\mathsf{V=L}$. If $\mathcal{N}= \langle N, \in^{\mathcal{N}} \rangle$ is such that $\mathcal{M} \prec_{e, 2} \mathcal{N}$, $\mathcal{N} \models \Delta_0\textsf{-Collection}$ and $\mathsf{Ord}^{\mathcal{N}} \backslash \mathsf{Ord}^{\mathcal{M}}$ is nonempty and contains no least new element, then $\mathcal{M} \models \Pi_2\textsf{-Collection}$.
\end{theorem} 

\begin{proof}
Let $\mathcal{N}= \langle N, \in^{\mathcal{N}} \rangle$ is such that $\mathcal{M} \prec_{e, 2} \mathcal{N}$, $\mathcal{N} \models \Delta_0\textsf{-Collection}$ and $\mathsf{Ord}^{\mathcal{N}} \backslash \mathsf{Ord}^{\mathcal{M}}$ is nonempty and contains no least new element. By Theorem \ref{Th:KPIInSigma2EndExtension}, $\mathcal{N} \models \mathsf{KPI}+\Sigma_1\textsf{-Foundation}+\mathsf{V=L}$. Using Lemma \ref{Th:NewCardinalInExtension}, let $\kappa \in \mathsf{Ord}^{\mathcal{N}} \backslash \mathsf{Ord}^{\mathcal{M}}$ be such that $\mathcal{N} \models (\kappa \textrm{ is a cardinal})$. Therefore, by Lemma \ref{Th:ElementarityOfLKappaInExtension}, $\mathcal{M} \prec_{e, 2} \langle (L_\kappa^{\mathcal{N}})^*, \in^{\mathcal{N}} \rangle$.  Moreover, $\mathsf{Ord}^{\langle (L_\kappa^{\mathcal{N}})^*, \in^{\mathcal{N}} \rangle} \backslash \mathsf{Ord}^{\mathcal{M}}$ is nonempty and contains no least element and, by Lemma \ref{Th:LKappaAdmissibleInExtension}, $\langle (L_\kappa^{\mathcal{N}})^*, \in^{\mathcal{N}} \rangle \models \mathsf{KPI}+\textsf{Separation}$. Therefore, by Theorem \ref{Th:SimpsonResult}, $\mathcal{M} \models \Pi_2\textsf{-Collection}$.    
\Square
\end{proof}

Using this result we show that the minimum model of $\mathsf{M}+\Pi_1\textsf{-Collection}$ has no proper $\Sigma_2$-elementary end extension that satisfies $\Delta_0\textsf{-Collection}$.

\begin{theorem} \label{Th:CounterexampleKaufmannQuestion}
Let $\alpha \in \omega_1$ be such that $L_\alpha$ is the minimum model of $\mathsf{M}+\Pi_1\textsf{-Collection}$. There is no $\mathcal{N}= \langle N, \in^{\mathcal{N}} \rangle$ such that $\langle L_\alpha, \in \rangle \prec_{e, 2} \mathcal{N}$, $L_\alpha \neq N$ and $\mathcal{N} \models \Delta_0\textsf{-Collection}$. 
\end{theorem} 

\begin{proof}
Suppose, for a contradiction, that $\mathcal{N}= \langle N, \in^{\mathcal{N}} \rangle$ is such that $\langle L_\alpha, \in \rangle \prec_{e, 2} \mathcal{N}$, $N \neq L_\alpha$ and $\mathcal{N} \models \Delta_0\textsf{-Collection}$. By Theorem \ref{Th:KPIInSigma2EndExtension},
\[
\mathcal{N} \models \mathsf{KPI}+\Sigma_1\textsf{-Foundation}+\mathsf{V=L}.
\]
Since $N \neq L_\alpha$, $\mathrm{Ord}^{\mathcal{N}} \backslash \mathsf{Ord}^{\langle L_\alpha, \in \rangle}$ is nonempty. Suppose, for a contradiction, that $\gamma$ is the least element of $\mathrm{Ord}^{\mathcal{N}} \backslash \mathsf{Ord}^{\langle L_\alpha, \in \rangle}$. Therefore, $(L_\gamma^{\mathcal{N}})^*= L_\alpha$, and 
\[
\mathcal{N} \models \exists x((x \textrm{ is transitive}) \land (\langle x, \in \rangle \models \mathsf{M}+\Pi_1\textsf{-Collection})).
\]
Since $\langle L_\alpha, \in \rangle \prec_{e, 1} \mathcal{N}$,
\[
\langle L_\alpha, \in \rangle \models \exists x((x \textrm{ is transitive}) \land (\langle x, \in \rangle \models \mathsf{M}+\Pi_1\textsf{-Collection})),
\]
but this contradicts the fact that $L_\alpha$ is the minimum model of $\mathsf{M}+\Pi_1\textsf{-Collection}$. Therefore, $\mathrm{Ord}^{\mathcal{N}} \backslash \mathsf{Ord}^{\langle L_\alpha, \in \rangle}$ is nonempty and has no least element. So, by Theorem \ref{Th:BaseCaseMainTheorem}, $\langle L_\alpha, \in \rangle \models \Pi_2\textsf{-Collection}$, which contradicts Corollary \ref{Th:SeparationOfMinimumModels}.
\Square
\end{proof}

This result shows that complexity of the scheme $\Delta_0\textsf{-Collection}$ computed by Lemma \ref{Th:ComplexityOfCollection} is best possible.

\begin{corollary}
Let $\alpha \in \omega_1$ be such that $L_\alpha$ is the minimum model of $\mathsf{M}+\Pi_1\textsf{-Collection}$. There is no collection of $\Sigma_3$-sentences, $\Gamma$, such that
\[
\langle L_\alpha, \in \rangle \models \Gamma \textrm{ and } \\mathsf{M}^-+\Gamma \vdash \Delta_0\textsf{-Collection}.  
\] 
\Square
\end{corollary}

\section[$\Sigma_{n+1}$-elementary end extensions]{$\Sigma_{n+1}$-elementary end extensions that satisfy $\Pi_{n-1}$-Collection} \label{Sec:GeneralCase}

Here we generalise the results of the preceding section to prove that for $n > 1$, if $\mathcal{M}$ satisfies $\mathsf{M}+\Pi_n\textsf{-Collection}+\mathsf{V=L}$ and $\mathcal{N}$ is such that $\mathcal{M} \prec_{e, n+1} \mathcal{N}$, $\mathcal{N}$ satisfies $\Pi_{n-1}\textsf{-Collection}$ and $\mathsf{Ord}^{\mathcal{N}} \backslash \mathsf{Ord}^{\mathcal{M}}$ is nonempty and contains no least element, then $\Pi_{n+1}\textsf{-Collection}$ holds in $\mathcal{M}$. This allows us to show that for all $n > 1$, the minimal model of $\mathsf{M}+\Pi_n\textsf{-Collection}$ has no proper $\Sigma_{n+1}$-elementary end extension that satisfies $\Pi_{n-1}\textsf{-Collection}$.

We begin by showing that if $\mathcal{M}$ satisfies $\mathsf{M}+\Pi_n\textsf{-Collection}+\mathsf{V=L}$ and $\mathcal{N}$ is a proper extension of $\mathcal{M}$ such that $\mathcal{M} \prec_{e, n+1} \mathcal{N}$ and $\mathcal{N}$ satisfies $\Pi_{n-1}\textsf{-Collection}$, then $\Sigma_{n+1}\textsf{-Separation}$ must hold in $\mathcal{M}$.

\begin{theorem} \label{Th:SeparationInBaseModel}
Let $n \in \omega$ with $n \geq 1$. Let $\mathcal{M}= \langle M, \in^{\mathcal{M}} \rangle$ be such that $\mathcal{M} \models \mathsf{M}+\Pi_n\textsf{-Collection}+\mathsf{V=L}$. If $\mathcal{N}= \langle N, \in^{\mathcal{N}} \rangle$ is such that $\mathcal{M} \prec_{e, n+1} \mathcal{N}$, $M \neq N$ and $\mathcal{N} \models \Pi_{n-1}\textsf{-Collection}$, then $\mathcal{M} \models \Sigma_{n+1}\textsf{-Separation}$.    
\end{theorem}

\begin{proof}
Let $\mathcal{N}= \langle N, \in^{\mathcal{N}} \rangle$ be such that $\mathcal{M} \prec_{e, n+1} \mathcal{N}$, $M \neq N$ and $\mathcal{N} \models \Pi_{n-1}\textsf{-Collection}$. By Theorems \ref{Th:KPIInSigma2EndExtension} and \ref{Th:TheoryInEndExtensionWithCollection}, $\mathcal{N} \models \mathsf{KPI}+\Pi_n \cup \Sigma_n\textsf{-Foundation}+\mathsf{V=L}$. Using Lemma \ref{Th:NewCardinalInExtension}, let $\kappa \in \mathsf{Ord}^{\mathcal{N}} \backslash \mathsf{Ord}^{\mathcal{M}}$ be such that $\mathcal{N} \models (\kappa \textrm{ is a cardinal})$. Note that, by Theorem \ref{Th:FoundationImpliesSeparation}, we need to show that $\Pi_{n+1}\textsf{-Foundation}$ holds in $\mathcal{M}$. Towards this end, let $\phi(y, x, \vec{z})$ be a $\Sigma_n$-formula and let $\vec{a} \in M$. Using the fact that $\mathcal{M} \prec_{e, n+1} \mathcal{N}$, for all $x \in M$,
\begin{equation} \label{eq:eq1}
\mathcal{M} \models \forall y \phi(y, x, \vec{a}) \textrm{ if and only if } \mathcal{N} \models (\forall y \in L_\kappa) \phi(y, x, \vec{a}).
\end{equation}
Now, the formula $(\forall y \in L_\kappa) \phi(y, x, \vec{z})$ is equivalent to a $\Sigma_n$-formula in $\mathcal{N}$. Suppose that $b \in M$ is an element of the class $A= \{ x \in M \mid \mathcal{M} \models \forall y \phi(y, x, \vec{a})\}$ of $\mathcal{M}$. Consider the class $B= \{ x \in N \mid \mathcal{N} \models (x \in b) \land(\forall y \in L_\kappa) \phi(y, x, \vec{a})\}$. By (\ref{eq:eq1}), if $B$ is empty, then $b$ is an $\in$-minimal element of $A$. Otherwise, (\ref{eq:eq1}) implies that any $\in$-minimal element of $B$ is also an $\in$-minimal element of $A$, and $\Sigma_n\textsf{-Foundation}$ in $\mathcal{N}$ ensures that $B$ must have an $\in$-minimal element. This shows that $\Pi_{n+1}\textsf{-Foundation}$ and thus $\Sigma_{n+1}\textsf{-Separation}$ holds in $\mathcal{M}$.
\Square  
\end{proof}

\begin{remark}
By Theorem \ref{Th:CollectionDoesNotProveFoundation}, there exists $\mathcal{M}= \langle M, \in^{\mathcal{M}}\rangle$ such that $\mathcal{M} \models \mathsf{M}+\Pi_n\textsf{-Collection}+\mathsf{V=L}$ and $\Pi_{n+1}\textsf{-Foundation}$ fails in $\mathcal{M}$. Theorem \ref{Th:SeparationInBaseModel} shows that such an $\mathcal{M}$ does not have a proper $\Sigma_{n+1}$-elementary end extension that satisfies $\Pi_{n-1}\textsf{-Collection}$. Therefore, Theorem \ref{Th:SeparationInBaseModel} already shows that there exists nonstandard models of $\mathsf{M}+\Pi_n\textsf{-Collection}$ that have no proper $\Sigma_{n+1}$-elementary end extension satisfying $\Pi_{n-1}\textsf{-Collection}$.   
\end{remark}

The following is a weakening of \cite[Corollary 1.27]{res87}, which is attributed to R. Pino, obtained by noting that, in the theory $\mathsf{KPI}$, $\Pi_{n+1}\textsf{-Foundation}$ implies principle of induction on the natural numbers for $\Sigma_{n+1}$-formulae.

\begin{lemma} \label{Th:ArbitrarilyLargeReflectingL}
(Pino) Let $n \in \omega$ with $n \geq 1$. The theory $\mathsf{KPI}+\Pi_n\textsf{-Collection}+\Pi_{n+1}\textsf{-Foundation}+\mathsf{V=L}$ proves that for all ordinals $\alpha$, there exists an ordinal $\beta > \alpha$ such that $L_\beta \prec_n \mathbb{V}$. \Square
\end{lemma} 

This allows us to prove the Main Theorem in the cases where $n \geq 2$.

\begin{theorem} \label{Th:GeneralCaseOfMainTheorem}
Let $n \in \omega$ with $n \geq 2$. Let $\mathcal{M}= \langle M, \in^{\mathcal{M}} \rangle$ be such that $\mathcal{M} \models \mathsf{M}+\Pi_n\textsf{-Collection}+\mathsf{V=L}$. If $\mathcal{N}= \langle N, \in^{\mathcal{N}} \rangle$ is such that $\mathcal{M} \prec_{e, n+1} \mathcal{N}$, $\mathcal{N} \models \Pi_{n-1}\textsf{-Collection}$ and $\mathsf{Ord}^{\mathcal{N}} \backslash \mathsf{Ord}^{\mathcal{M}}$ is nonempty and contains no least element, then $\mathcal{M} \models \Pi_{n+1}\textsf{-Collection}$.
\end{theorem}

\begin{proof}
Let $\mathcal{N}= \langle N, \in^{\mathcal{N}} \rangle$ be such that $\mathcal{M} \prec_{e, n+1} \mathcal{N}$, $\mathcal{N} \models \Pi_{n-1}\textsf{-Collection}$ and $\mathsf{Ord}^{\mathcal{N}} \backslash \mathsf{Ord}^{\mathcal{M}}$ is nonempty and contains no least element. By Theorem \ref{Th:TheoryInEndExtensionWithCollection},
\[
\mathcal{N} \models \mathsf{M}+\Pi_{n-1}\textsf{-Collection}+\Sigma_n\textsf{-Separation}+\mathsf{V=L}.
\]
Using Lemma \ref{Th:NewCardinalInExtension}, let $\kappa \in \mathsf{Ord}^{\mathcal{N}} \backslash \mathsf{Ord}^{\mathcal{M}}$ be such that $\mathcal{N} \models (\kappa \textrm{ is a cardinal})$.

Work inside $\mathcal{N}$. Consider
\[
A= \{ X \in L_\kappa \mid \exists \beta(X= L_\beta) \land (X \prec_n \mathbb{V})\}, 
\]
which is a set by $\Sigma_n\textsf{-Separation}$. Let $B= \{ \xi \in \kappa \mid (\exists x \in A)(\exists \beta \in x)(\beta \in \kappa \land \xi \leq \beta)\}$. Let $\eta= \sup(B)$. 

Work in the meta-theory again. By Lemma \ref{Th:ArbitrarilyLargeReflectingL}, $\mathsf{Ord}^{\mathcal{M}} \subseteq B^*$. And, $B^* \cap (\mathsf{Ord}^{\mathcal{N}} \backslash \mathsf{Ord}^{\mathcal{M}}) \neq \emptyset$, otherwise $\eta$ would be the least element of $\mathsf{Ord}^{\mathcal{N}} \backslash \mathsf{Ord}^{\mathcal{M}}$. Let $\alpha \in \mathsf{Ord}^{\mathcal{N}} \backslash \mathsf{Ord}^{\mathcal{M}}$ be such that $L_\alpha^{\mathcal{N}} \in A^*$. Therefore,
\[
\mathcal{M} \prec_{e, n} \langle (L_\alpha^{\mathcal{N}})^*, \in^{\mathcal{N}} \rangle \prec_{e, n} \mathcal{N}.
\]
Let $\phi(x, \vec{z})$ be a $\Pi_n$-formulae and let $\vec{a} \in M$. Suppose that $\langle (L_\alpha^{\mathcal{N}})^*, \in^{\mathcal{N}} \rangle \models \exists x \phi(x, \vec{a})$. Let $b \in (L_\alpha^{\mathcal{N}})$ be such that $\langle (L_\alpha^{\mathcal{N}})^*, \in^{\mathcal{N}} \rangle \models \phi(b, \vec{a})$. So, $\mathcal{N} \models \exists x \phi(x, \vec{a})$. And, since $\mathcal{M} \prec_{e, n+1} \mathcal{N}$, $\mathcal{M} \models \exists x \phi(x, \vec{a})$. This shows that $\mathcal{M} \prec_{e, n+1} \langle (L_\alpha^{\mathcal{N}})^*, \in^{\mathcal{N}} \rangle$. Note that, since $n \geq 2$ and $L_\alpha^{\mathcal{N}}$ is a set in $\mathcal{N}$,
\[
\langle (L_\alpha^{\mathcal{N}})^*, \in^{\mathcal{N}} \rangle \models \mathsf{KPI}+\textsf{Separation}.
\]
Therefore, by Theorem \ref{Th:SimpsonResult}, $\mathcal{M} \models \Pi_{n+1}\textsf{-Collection}$.       	
\Square
\end{proof}

Theorem \ref{Th:GeneralCaseOfMainTheorem} allows us to generalise Theorem \ref{Th:CounterexampleKaufmannQuestion} and provide a negative answer to Question \ref{Q:GeneralisedKaufmannCloteQuestion} by showing that for $n \geq 2$, the minimum model of $\mathsf{M}+\Pi_n\textsf{-Collection}$ has no proper $\Sigma_{n+1}$-elementary end extension that satisfies $\Pi_{n-1}\textsf{-Collection}$.

\begin{theorem}
Let $n \in \omega$ with $n \geq 2$. Let $\alpha \in \omega_1$ be such that $L_\alpha$ is the minimum model of $\mathsf{M}+\Pi_n\textsf{-Collection}$. There is no $\mathcal{N}= \langle N, \in^{\mathcal{N}} \rangle$ such that $\langle L_\alpha, \in \rangle \prec_{e, n+1} \mathcal{N}$, $N \neq L_\alpha$ and $\mathcal{N} \models \Pi_{n-1}\textsf{-Collection}$.
\end{theorem}

\begin{proof}
Suppose, for a contradiction, that $\mathcal{N}= \langle N, \in^{\mathcal{N}} \rangle$ is such that $\langle L_\alpha, \in \rangle \prec_{e, n+1} \mathcal{N}$, $N \neq L_\alpha$ and $\mathcal{N} \models \Pi_{n-1}\textsf{-Collection}$. By Theorem \ref{Th:TheoryInEndExtensionWithCollection},
\[
\mathcal{N} \models \mathsf{M}+\Pi_{n-1}\textsf{-Collection}+\Sigma_n\textsf{-Separation}+\mathsf{V=L}
\]
Since $N\neq L_\alpha$, $\mathsf{Ord}^{\mathcal{N}} \backslash \mathsf{Ord}^{\langle L_\alpha, \in \rangle}$ is nonempty. To be able to apply Theorem \ref{Th:SimpsonResult}, we need to show that $\mathsf{Ord}^{\mathcal{N}} \backslash \mathsf{Ord}^{\langle L_\alpha, \in \rangle}$ has no least element. Suppose, for a contradiction, that $\gamma$ is the least element of $\mathsf{Ord}^{\mathcal{N}} \backslash \mathsf{Ord}^{\langle L_\alpha, \in \rangle}$. But then $(L_\gamma^{\mathcal{N}})^*= L_\alpha$ and 
\[
\mathcal{N} \models \exists x((x \textrm{ is transitive}) \land (\langle x, \in \rangle \models \mathsf{M}+\Pi_n\textsf{-Collection})).
\]
Since $\langle L_\alpha, \in \rangle \prec_{e, 1} \mathcal{N}$,
\[
\langle L_\alpha, \in \rangle \models \exists x((x \textrm{ is transitive}) \land (\langle x, \in \rangle \models \mathsf{M}+\Pi_n\textsf{-Collection})),
\]
which contradicts the fact that $L_\alpha$ is the minimum model of $\mathsf{M}+\Pi_n\textsf{-Collection}$. This shows that $\mathsf{Ord}^{\mathcal{N}} \backslash \mathsf{Ord}^{\langle L_\alpha, \in \rangle}$ is nonempty and contains no least element. Therefore, by Theorem \ref{Th:SimpsonResult}, $\langle L_\alpha, \in \rangle \models \Pi_{n+1}\textsf{-Collection}$, which contradicts Corollary \ref{Th:SeparationOfMinimumModels}.
\Square 
\end{proof}

This shows that the complexity of $\Pi_n\textsf{-Collection}$ in the L\'{e}vy hierarchy computed by Lemma \ref{Th:ComplexityOfCollection} is best possible.

\begin{corollary}
Let $n \in \omega$ with $n \geq 1$. Let $\alpha \in \omega_1$ be such that $L_\alpha$ is the minimum model of $\mathsf{M}+\Pi_{n+1}\textsf{-Collection}$. Then there is no collection of $\Sigma_{n+3}$-sentences, $\Gamma$, such that
\[
\langle L_\alpha, \in \rangle \models \Gamma \textrm{ and } \mathsf{M}+\Gamma \vdash \Pi_{n}\textsf{-Collection}.
\]
\Square
\end{corollary}

\cite[Remark 2]{kau81} observes that there is a countable ordinal $\alpha$ such that $\langle L_\alpha, \in \rangle$ has proper $\Sigma_1$-elementary end extension that satisfies $\mathsf{KPI}$, but $\Pi_1\textsf{-Collection}$ does not hold in $\langle L_\alpha, \in \rangle$. This observation can be generalised to show that the assumption that the end extension contains no least new ordinal can not be eliminated from Theorem \ref{Th:MainTheorem}.

\begin{observation}
There exists $\alpha, \beta \in \omega_1$ such that $\alpha < \beta$, $\langle L_\alpha, \in \rangle \prec_{e, n+1} \langle L_\beta, \in \rangle$,\\ $\langle L_\beta, \in \rangle \models \Pi_{n+1}\textsf{-Collection}$ and $\Pi_{n+1}\textsf{-Collection}$ does not hold in $\langle L_\alpha, \in \rangle$.
\end{observation}

\begin{proof}
Let $\beta \in \omega_1$ be such that $\langle L_\beta, \in \rangle \models \mathsf{M}+\Pi_{n+1}\textsf{-Collection}$. By Lemma  \ref{Th:ArbitrarilyLargeReflectingL}, the set $A= \{\gamma \in \omega_1 \mid \langle L_\gamma, \in \rangle \prec_{e, n+1} \langle L_\beta, \in \rangle \}$ is nonempty. Let $\alpha$ be the least element of $A$. Then, by Lemma \ref{Th:ArbitrarilyLargeReflectingL}, $\langle L_\alpha, \in \rangle$ does not satisfy $\Pi_{n+1}\textsf{-Collection}$.  
\Square
\end{proof}

Finally, we observe that Theorem \ref{Th:MainTheorem} yields a characterisation of the countable models of $\mathsf{MOST}+\mathsf{V=L}$ have a $\Sigma_{n+1}$-elementary end extension that satisfies $\Pi_{n-1}\textsf{-Collection}$ and contains a new ordinal but no least new ordinal. The following is a strengthening of \cite[Theorem 4.15]{mck25} can be obtained by replacing the crude compactness argument used in the proof \cite[Theorem 4.5]{mck25} with a more sophisticated Henkin model construction such a argument used in the proof of \cite[Theorem 2.2]{fri73}:

\begin{theorem} \label{Th:EndExtensionFromBarwiseCompactness}
Let $n \in \omega$ with $n \geq 1$. Let $S$ be a recursively enumerable $\mathcal{L}$-theory such that
\[
S \vdash \mathsf{KP}+\Pi_n\textsf{-Collection}+\Sigma_{n+1}\textsf{-Foundation},
\]
and let $\mathcal{M}= \langle M, \in^{\mathcal{M}} \rangle$ be a countable model of $S$. Then there exists $\mathcal{N}= \langle N, \in^{\mathcal{N}} \rangle$ such that $\mathcal{M} \prec_{e, n} \mathcal{M}$, $\mathcal{N} \models S$ and $\mathsf{Ord}^{\mathcal{N}} \backslash \mathsf{Ord}^{\mathcal{M}}$ is nonempty and contains no least element. \Square
\end{theorem}

This allows us to show that a countable model, $\mathcal{M}$, of $\mathsf{MOST}+\mathsf{V=L}$ has a $\Sigma_n$-elementary end extension that satisfies $\Pi_{n-1}\textsf{-Collection}$ and contains a new ordinal but no least new ordinal if and only if $\mathcal{M}$ satisfies $\Pi_{n+1}\textsf{-Collection}$.

\begin{theorem}
Let $n \in \omega$ with $n \geq 1$. Let $\mathcal{M}= \langle M, \in^{\mathcal{M}} \rangle$ be a countable model of $\mathsf{MOST}+\mathsf{V=L}$. Then the following are equivalent:
\begin{itemize}
\item[(I)] $\mathcal{M} \models \Pi_{n+1}\textsf{-Collection}$;
\item[(II)] there exists $\mathcal{N}= \langle N, \in^{\mathcal{N}} \rangle$ such that $\mathcal{M} \prec_{e, n+1} \mathcal{N}$, $\mathcal{N} \models \Pi_{n-1}\textsf{-Collection}$ and $\mathsf{Ord}^{\mathcal{N}} \backslash \mathsf{Ord}^{\mathcal{M}}$ is nonempty and contains no least new element.
\end{itemize}
\end{theorem}

\begin{proof}
$(II) \Rightarrow (I)$ follows immediately from Theorems \ref{Th:KaufmannTheorem} and \ref{Th:MainTheorem}. To see that $(I) \Rightarrow (II)$, using Theorem \ref{Th:KaufmannTheorem}, let $\mathcal{N}= \langle N, \in^{\mathcal{N}} \rangle$ be such that $\mathcal{M} \prec_{e, n+2} \mathcal{N}$ and $M \neq N$. By Theorem \ref{Th:TheoryInEndExtensionWithoutCollection}, $\mathcal{N} \models \Pi_{n-1}\textsf{-Collection}$. Therefore, if $\mathsf{Ord}^{\mathcal{N}} \backslash \mathsf{Ord}^{\mathcal{M}}$ contains no least element, then we are done. So, assume that $\mathsf{Ord}^{\mathcal{N}} \backslash \mathsf{Ord}^{\mathcal{M}}$ has a least element. In this case, $\mathcal{M} \models \mathsf{Foundation}$ and, using Theorem \ref{Th:EndExtensionFromBarwiseCompactness}, there exists $\mathcal{N}_0= \langle N_0, \in^{\mathcal{N}_0} \rangle$ such that $\mathcal{M} \prec_{e, n+1} \mathcal{N}_0$, $\mathcal{N}_0 \models \Pi_{n+1}\textsf{-Collection}$ and $\mathsf{Ord}^{\mathcal{N}_0} \backslash \mathsf{Ord}^{\mathcal{M}}$ is nonempty and contains no least element. 
\Square
\end{proof}

\section[Questions]{Questions}

The Main Theorem assumes that \textsf{Powerset} holds in the structure being end extended, and the proof presented here uses this assumption.

\begin{question} \label{Q:NecessityOfPowerset}
Let $n \in \omega$ with $n \geq 1$. Let $\mathcal{M} \models \mathsf{KPI}+\Pi_n\textsf{-Collection}+\mathsf{V=L}$. If $\mathcal{N}$ is a $\Sigma_{n+1}$-elementary end extension of $\mathcal{M}$ that satisfies $\Pi_{n-1}\textsf{-Collection}$ and contains a new ordinal but no least new ordinal, does $\Pi_{n+1}\textsf{-Collection}$ necessarily hold in $\mathcal{M}$?
\end{question}

If Question \ref{Q:NecessityOfPowerset} has a negative answer, can a counterexample be found amongst the minimum models?

\begin{question}
Let $n \in \omega$ with $n \geq 1$. Does the minimum model of $\mathsf{KPI}+\Pi_n\textsf{-Collection}$ have a proper $\Sigma_{n+1}$-elementary end extension that satisfies $\Pi_{n-1}\textsf{-Collection}$?
\end{question}

In \cite{mck25}, it is shown that every countable model of $\mathsf{KP}+\Pi_n\textsf{-Collection}+\Sigma_{n+1}\textsf{-Foundation}$ has a proper elementary end extension that satisfies $\mathsf{KP}+\Pi_n\textsf{-Collection}+\Sigma_{n+1}\textsf{-Foundation}$. The following two questions from \cite{mck25} remain open:

\begin{question} \label{Q:OldQuestion}
Does every countable model of $\mathsf{KPI}+\Pi_n\textsf{-Collection}$ have a proper $\Sigma_n$-elementary end extension that satisfies $\mathsf{KPI}+\Pi_n\textsf{-Collection}$?
\end{question}

And, if Question \ref{Q:OldQuestion} has a negative answer:

\begin{question} \label{Q:OldQuestion2}
Does every countable model of $\mathsf{M}+\Pi_n\textsf{-Collection}$ have a proper $\Sigma_n$-elementary end extension that satisfies $\mathsf{M}+\Pi_n\textsf{-Collection}$?
\end{question}

\bibliographystyle{alpha}
\bibliography{.}           

\end{document}